\documentclass{article}
\usepackage{algorithm}
\usepackage{algpseudocode}
\usepackage{nicematrix}
\usepackage[a4paper, total={6.5in, 9in}]{geometry}
\usepackage[
backend=biber,
style=alphabetic,
sorting=ynt
]{biblatex}
\usepackage{hyperref}
\hypersetup{
    colorlinks=true,
    linkcolor=blue,
    urlcolor=blue,
    citecolor=blue,
    pdftitle={A saturation theorem for submonoids of nilpotent groups and the Identity Problem},
    }
\usepackage{amsmath,amsthm,amssymb}

\newtheorem{Proposition}{Proposition}
\newtheorem{Theorem}{Theorem}
\newtheorem{Corollary}{Corollary}
\newcommand{\lcs}[1]{\gamma_{#1}(G)}
\newcommand{\UT}[2]{UT_{#1}(#2)}
\newcommand{\Heisen}[2]{H_{#1}(#2)}

\title{A saturation theorem for submonoids of nilpotent groups and the Identity Problem}
\author{Doron Shafrir \\
{\small doron.abc@gmail.com}}
\date{11 February 2024}

\begin{document}

\maketitle

\begin{abstract}
If $M$ is a submonoid of a f.g.\@ nilpotent group $G$, and $MG'$ is a finite index subgroup of $G$, then $M$ itself is a finite index subgroup of $G$. If $MG'=G$, then $M=G$. This generalizes a well-known theorem for subgroups of f.g.\@ nilpotent groups. As a result, we give an algorithm for the identity problem in nilpotent groups.
\end{abstract}

\section{Introduction}
Given upper unitriangular matrices $B_1,...,B_n\in\UT{n}{\mathbb{Z}}$, is there a product of the $B_i$'s (allowing for repetitions) that equals the identity matrix $I_n$? In this work we present an efficient algorithm for this problem.
More generally, the identity problem for a f.g.\@ group $G$ is the following decision problem: Given a finite subset $A\subset G$, is  $e_G\in A^+$? The elements of $A$ are given as words in a fixed generating set $X$ of $G$. We show an algorithm for this problem when $G$ is nilpotent. The matrix version is a special case, since any $A\in\UT{n}{\mathbb{Z}}$ can be easily expressed as a word in the generating set of the elementary matrices $X=\{e_{ij}\mid i<j\}\subset \UT{n}{\mathbb{Z}}$\footnote{ The matrices $e_{ij}$  form a Mal’cev basis of $\UT{n}{\mathbb{Z}}$ \cite[Lemma 7.2]{clement2017theory}. The word representation will use similar space to the matrix representation, as long as compressed representation is used \cite{macdonald2015logspace}}

\subsection{A saturation theorem}
It is well-known that if $G$ is a f.g.\@ nilpotent group, and $H\le G$ is a subgroup such that $HG'$ is a finite index subgroup of $G$, then in fact $H$ itself is a finite index subgroup of $G$. Also, if $HG'=G$, then $H=G$. We show that the same holds even if $H$ is a submonoid and not assumed to be closed under inverses a-priori.

\subsection{Related results}
Bell and Potapov \cite{bell2010undecidability} showed that the Identity problem is undecidable for $SL_4(\mathbb{Z})$. Ko, Niskanen, and Potapov \cite{ko_et_al} showed it's decidable for the Heisenberg groups $\Heisen{n}{\mathbb{Q}}$. Dong \cite{dongUT4} proved decidability for $\UT{4}{\mathbb{Q}}$, and extended it to nilpotent groups of class at most 10 \cite{dong2024boundedclass}, with assistance of computer algebra software. The Identity Problem is contained in the Submonoid Memebership Problem, which Romankov showed to be undecidable for nilpotent groups of class 2 \cite{romankov2022undecidability}. While preparing this paper, a similar saturation theorem was announced in an arXiv preprint by Bodart \cite[Proposition 2.5]{bodart2024membership}, to appear in a future joint paper by Bodart, Ciobanu, and Metcalf\footnote{I have wrote to Markus Lohrey and Ruiwen Dong a description of the algorithm and Corollary~\ref{corr:saturation} on December 7th, 2023. I was informed by Markus Lohrey about Bodart's preprint when it was uploaded.}.

\section{Notations} \label{sec:notations}
Given a group $G$, the lower central series of $G$ is defined by $\lcs{0}=G$, $\lcs{n+1}=[G,\lcs{n}]$. We also use $G'=\lcs{1}=[G,G]$.  $G$ is nilpotent of class $c$ if $\lcs{c}=\{e_G\}$ but $\lcs{c-1}\ne\{e_G\}$.
Given a set $A\subseteq G$, we denote $\langle A\rangle,A^+,A^*$ the subgroup, subsemigroup and submonoid generated by $A$ respectively.
If we write $v>\alpha$ when $v\in\mathbb{R}^n,\alpha\in\mathbb{R}$ , we mean inequality at every coordinate.
In the algorithm, If $A,B$ are finite sets, we identify $v\in \mathbb{R}^A=\{v:A\rightarrow\mathbb{R}\}$ with vectors $\mathbb{R}^{|A|}$, and $M\in\mathbb{R}^{A\times B}$ with  $|A|\times|B|$ matrices, where enumerations of $A,B$ are implicitly assumed. The product $Mv\in\mathbb{R}^B$ is also defined naturally.
If $A=\{a_1,\ldots,a_n\}, I\subseteq\{1,\ldots,n\}$ then $A_I=\{a_i\mid i\in I\}$.

\section{A saturation theorem for submonoids of nilpotent groups}
We first recall that the commutator map is bilinear when the result is in the center of the group:
\begin{Proposition}
\label{prop:bilinear}
if $x_1,x_2,y\in G$ and $[x_1,y]\in Z(G)$ then $[x_1x_2,y]=[x_1,y][x_2,y]$ and $[y,x_1x_2]=[y,x_1][y,x_2]$
\end{Proposition}
\begin{proof}
\begin{gather*}
[x_1x_2,y] = (x_1x_2)^{-1}y^{-1}x_1x_2y=x_2^{-1}x_1^{-1}y^{-1}x_1(yy^{-1})x_2y=\\
=x_2^{-1}[x_1,y]y^{-1}x_2y=[x_1,y]x_2^{-1}y^{-1}x_2y=[x_1,y][x_2,y]
\end{gather*}
The second identity can be proven like the first, or alternatively by using $[y,x]=[x,y]^{-1}$.
\end{proof}
As a corollary, if $G$ is nilpotent of class $c$, the commutator map $G\times\lcs{c-2}\rightarrow\lcs{c-1}$ is bilinear, and its image generates $\lcs{c-1}$ by definition of the lower central series (see \cite[ Theorem 2.15]{clement2017theory}). We now recall the proof of the well-known result for groups, which we will generalize.
\begin{Theorem}\label{thm:groups}
Let $G$ be a finitely generated nilpotent group, $H\le G$ a subgroup such that $[G:HG']<\infty$. Then, $[G:H]<\infty$. Moreover, if $HG'=G$ then $H=G$.
\end{Theorem}
\begin{proof}
We use induction on the nilpotency class $c$ of $G$. If $c=1$, $G'=\{1\}$ and there is nothing to prove. Now, assume $c>1$. We define $Z=\lcs{c-1}\subseteq Z(G)$. By the induction hypothesis on $HZ/Z$ and  $G/Z$,  we have $[G/Z:HZ/Z]=[G:HZ]<\infty$, so there is some $e$ such that $G^e\subseteq HZ$. Let $x\in G,y\in \lcs{c-2}$. We have $x^e,y^e\in HZ$, so there are $a,b\in Z$ such that $x^ea,y^eb\in H$. We have $[x^ea,y^eb]\in H$ and also $$[x^ea,y^eb]=[x^e,y^e][x^e,b][a,x^e][a,b]=[x^e,y^e]=[x,y]^{e^2}$$
Where we used $a,b,[x,y]\in Z(G)$ and Proposition~\ref{prop:bilinear}. Now, since by definition $Z=[G,\lcs{c-2}]$, we have $Z^{e^2}\subseteq H$.
Since $Z$ is a finitely generated Abelian group, we have $[Z:H\cap Z]<\infty$. Finally, $[G:H]=[G:HZ][HZ:H]=[G:HZ][Z:H\cap Z]<\infty$.

For the second statement, note that if $HG'=1$, we can use $e=1$ throughout the proof, and get $G=HZ,Z\subseteq H$ which implies $H=G$.
\end{proof}
\begin{Proposition}\label{prop:grpgrp}
Assume $G$ is a group, $N\triangleleft G$ a normal subgroup, and $M\le G$  a submonoid of $G$ such that both $M\cap N$ and  $MN$ are closed under inverses, that is, they are both groups. Then, $M$ is also a subgroup of $G$. 
\end{Proposition}
\begin{proof}
Let $m\in M$, we need to prove $m^{-1}\in M$. Since $MN$ is a group, for some $n\in N$ we have $m^{-1}n\in M$.
Therefore $n=m(m^{-1}n)\in M\cap N$, so also $n^{-1}\in M$ by assumption on $M\cap N$, so $m^{-1}=(m^{-1}n)n^{-1}\in M$ as needed.
\end{proof}
The following result is at the heart of our proof:
\begin{Proposition}
\label{prop:main}
Let $G$ be any group, $M\le G$ a submonoid,  $z=[x,y]\in Z(G)$. If $M\langle z\rangle$ is a finite index subgroup of $G$ then $M$ itself is a finite index subgroup of $G$.
\end{Proposition}
\begin{proof}
Denote $Z=\langle z\rangle$. Since $MZ$ is a finite index subgroup of $G$, by Proposition~\ref{prop:grpgrp} it is enough to prove $M\cap Z$ is a finite index subgroup of $Z$. If $Z$ is finite, this is trivial, so assume $Z\simeq\mathbb{Z}$. Since $[G:MZ]<\infty$  there is $e\in\mathbb{N}$ such that $x^e,y^e\in MZ$, so for some $a,b\in Z$ we have $x^ea,y^eb\in M$.  By Proposition~\ref{prop:bilinear}, using $a,b,[x,y]\in Z(G)$, 
$$[x^ea,y^eb]=[x,y]^{e^2}[x^e,b][a,x^e][a,b]=z^{e^2}$$
Replacing $x,y$ with $x^ea,y^eb$ we can now assume $x,y\in M$ and $[x,y]=z^{e^2}$.
Since $MZ$ is a subgroup of $G$ and therefore closed under inverses, for some  $i,j\in\mathbb{Z}$ we have $x^{-1}z^i,y^{-1}z^j\in M$.
The elements $x^{-1}z^i,y^{-1}z^j$ can be thought of as approximation to $x^{-1},y^{-1}$ inside $M$. Since $x,y,x^{-1}z^i,y^{-1}z^j$ are all in $M$, for every $N\in\mathbb{N}$,
$$(x^{-1}z^i)^N(y^{-1}z^j)^Nx^Ny^N\in M$$
On the other hand, using Proposition~\ref{prop:bilinear} and $z\in Z(G)$,
$$(x^{-1}z^i)^N(y^{-1}z^j)^Nx^Ny^N=z^{(i+j)N}[x^N,y^N]=z^{(i+j)N}[x,y]^{N^2}=z^{(i+j)N+e^2N^2}$$
It is clear that for large enough $N$,  $(i+j)N+e^2N^2>0$. 
Similarly, $$(y^{-1}z^j)^N(x^{-1}z^i)^Ny^Nx^N=z^{(i+j)N-e^2N^2}\in M$$
And the exponent in negative for large enough $N$. Now, $M\cap Z$ is a submonoid of $Z$ that includes both positive and negative powers of $z$, therefore, it is a finite index subgroup of $Z$, as needed.
\end{proof}

We can now state our main theorem:
\begin{Theorem}
\label{thm:main}
Assume $G$ is a finitely generated nilpotent group and $M\le G$ a submonoid of $G$. If $MG'$ is a finite index subgroup of $G$, so is $M$ itself.
\end{Theorem}
\begin{proof}
We use induction on the nilpotency class $c$ of $G$. For $c=1$ there is nothing to prove. Assume $c>1$. Define $Z=\lcs{c-1}$. By induction, $MZ$ is a finite index subgroup of $G$. We choose $z_1,z_2,\ldots,z_n\in Z\cap\{[x,y]\mid x\in G,y\in\lcs{c-2}\} $ generating $Z$. Note that we make sure all $z_i$ are commutators for Proposition~\ref{prop:main} to be applicable. Set$$M_0=M,M_1=M\langle z_1\rangle, M_2=M\langle z_1, z_2\rangle,\ldots M_n=M\langle z_1,\ldots z_n\rangle=MZ$$
we have $M_i\langle z_{i+1}\rangle=M_{i+1}$. Since $M_n=MZ=M_{n-1}\langle z_n\rangle$ is a subgroup of finite index of $G$, by Proposition~\ref{prop:main} so is $M_{n-1}=M_{n-2}\langle z_{n-1}\rangle$, and continuing this way so are $M_{n-2},M_{n-3}...$ until we reach $M_0=M$.
\end{proof}
\begin{Corollary}\label{corr:saturation}
Let $G$ be f.g.\@ and nilpotent, $M\le G$ a submonoid, and $MG'=G$. Then $M=G$.
\end{Corollary}
\begin{proof}
By Theorem~\ref{thm:main}, $M$ is a subgroup of $G$. Now we are in the classical case, see Theorem~\ref{thm:groups} or \cite[Theorem 7.18]{clement2017theory}.
\end{proof}
\section{The Identity Problem for nilpotent groups}
First, we need a simple result from linear programming.
\begin{Proposition}
\label{prop:LPduality}
Given $A\in\mathbb{R}^{m\times n}$ such that $$\{v\in\mathbb{R}^n\mid Av=0,\sum_{i=1}^n v_i=1,v\ge 0\}=\varnothing$$ Then, there is some $u\in\mathbb{Z}^m$ with $A^tu>0$.
\end{Proposition}
\begin{proof}
Define $A_1=(A\mid\boldsymbol{1}^m)\in\mathbb{R}^{(m+1)\times n}$ and $b=(0,0,..,0,1)^t\in\mathbb{R}^{m+1}$. We know that $\{A_1v=b,v\ge 0\}$ has no solution, so by Farkas lemma (\cite{farkas_dax1997classroom}), there is $u_1=(u\mid \alpha)^t\in\mathbb{R}^{m+1}$ such that $A_1^tu_1=A^tu+\alpha\boldsymbol{1}^m\ge 0$ and $b^tu_1=\alpha<0$, implying  $A^tu\ge-\alpha>0$. Since $\{u\mid Au>0\}$ is open, we can perturb $u$ so $u\in\mathbb{Q}^m$. Finally, we multiply by a common denominator to get $u\in\mathbb{Z}^m$.
\end{proof}

Now we present the algorithm. We actually solve a more general problem than the identity problem: Given $A\subset G$, find  $A_{inv}=\{g\in A\mid g^{-1}\in A^*\}$. Equivalently, $A_{inv}$ is the set of $a\in A$ which are part of a word in the $a_i$'s with value $e_G$. Clearly, $e_G\in A^+$ iff $A_{inv}\ne\varnothing$. We explain the meaning of each step in the algorithm, and then prove its correctness.

\algnewcommand{\TRUE}{\textbf{true}}
\algnewcommand{\FALSE}{\textbf{false}}
\algnewcommand{\NULL}{\textbf{null}}
\newcommand{\vars}{\texttt}
\begin{algorithm}
\caption{A nilpotent group $G$ is fixed. Given $A=\{a_1,...,a_n\}\subset G$, return  $A_{inv}$}
\begin{algorithmic}[1]\label{algorithm:main}
    \Function{FindInvertibleSubset}{$\vars{A}$} 
    \State $I \gets \{1,2,...,n\}$
    \While{$\vars{I}\ne\varnothing$}
        \State $\vars{R} \gets$  \Call{SubgroupRelations}{G,$\vars{A}_\vars{I}$} \Comment{Relations for a presentation of $\langle \vars{A}_\vars{I}\rangle \le G$}
        \State $\vars{M} \gets$ $\mathbb{Z}^{|\vars{R}|\times |\vars{I}|}$  matrix of signed occurrences of each $a\in\vars{A}_\vars{I}$ in each $r\in\vars{R}$
        \State $\vars{v} \gets $ \Call{Solve}{$\{v\in\mathbb{R}^{|\vars{I}|},\vars{M}v=0, \sum v=1, v\ge 0\}$}
        \If{$\vars{v}=\NULL$} \Comment{If no solution exists}
                \State \Return $\vars{A}_\vars{I}$
            \Else
                \State $\vars{I} \gets \{i\in \vars{I}\mid \vars{v}_i=0\}$
            \EndIf
        \EndWhile
\State \Return $\varnothing$ 
\EndFunction
\end{algorithmic}
\end{algorithm}

\newcommand{\algoname}[1]{\textnormal{\textsc{#1}}}

\emph{step 4.} We use the algorithm from \cite[Theorem 3.11]{macdonald2015logspace} (see also \cite[Theorem 3.4]{baumslag1991algorithmic}): Given a finite presentation of a nilpotent group $G$ and a finite subset $A\subset G$ (given as words in $X^\pm$), the algorithm finds a finite presentation for $\langle A\rangle$ in the generators $A$. Here a presentation $G$ is fixed and "hard coded" as a parameter of $\algoname{SubgroupRelations}$, but could also be an input to $\algoname{FindInvertibleSubset}$: $\algoname{SubgroupRelations}$ and therefore $\algoname{FindInvertibleSubset}$ is polynomial time as long as the nilpotency class of $G$ is fixed. Also, one can use $\langle A_{I_{prev}}\mid R\rangle$, with $I_{prev}$ being the last value of $I$, instead of $G$ as the containing group. 

\emph{step 5.} We count the occurrences of each generator of $A_I$ in each relation of $R$, with appropriate signs. For example, if $I=\{1,7,8\}$ and $R=\{a_8^2a_7^3, a_1^3a_7^{-6}a_1^6a_7\}$, then $M=\left(\begin{smallmatrix}0 & 3 & 2\\9 & -5 & 0\end{smallmatrix}\right)$ 

\emph{step 6.}  We call a Linear Programming solver, that returns any solution if it exists, and $\NULL$ if there is no solution.

\emph{step 10.} Here we consider $\vars{v}$ as function $\vars{v}:I\rightarrow\mathbb{R}$, see Section \ref{sec:notations}.

We now prove that the algorithm returns $A_{inv}$. First we show that any $a_i$ removed in step 10 indeed satisfies $a_i^{-1}\notin A^*$.  Assume $\vars{v}\ne\NULL$. Let $F_I$ be the free group with formal generators $\{a_i\mid i\in\vars{I}\}$, and define a homomorphism $f:F_I\rightarrow\mathbb{R}$ by $f(a_i)=v_i$. Since $\vars{Mv}=0$, we have $f(r)=0$ for any $r\in\vars{R}$, so $f$ factors through the quotient map in the presentation of $\langle A_I\rangle$ and we get $f:\langle A_I\rangle\rightarrow\mathbb{R}$, and $f(a_i)\ge 0$ since $\vars{v}\ge 0$. If a word $w$ in the $ A_I$ equals $e_G$, then in particular $f(w)=0$. On the other hand $f(a_i)\ge 0$ for all $i\in I$ implies $f(w)\ge 0$, and equality holds iff $w$ does not include any $a_i$ for which $v_i>0$. Therefore removing the elements is justified. Note that $v\ne\boldsymbol{0}$ by the condition $\sum_iv_i=1$ so  $I$ gets strictly smaller.

Next, assume $\vars{v}=\NULL$, so no solution exists. By Proposition~\ref{prop:LPduality}, there is some $u\in\mathbb{Z}^{|R|}$ such that $M^tu>0$. Set $w=M^tu\in\mathbb{Z}^I$, then $w>0$ (recall we identify $\mathbb{Z}^I$ with $\mathbb{Z}^{|I|}$ implicitly). Set $H=\langle A_I\rangle\le G$. Since each $r_i$ is a relation in $H$, the product $r=r_1^{u_1}\cdots r_{|R|}^{u_{|R|}}$ is also a relation in $H$, and thus also in the Abelianization $H/H'$. But in $H/H'$ the $a_i$'s are allowed to commute, so we can group the $a_i$'s in $r$  to get $\prod_{i\in I} a_i^{w_i}\in H'$ (here we used the definition of $\vars{M}$). Since $w_i>0$, we conclude that for every $i\in I$, $a_i^{-1}\in A_I^*H'$.  Therefore $ A_I^* H'= \langle A_I\rangle H'=H\cdot H'=H$. By corollary~\ref{corr:saturation}, we have $ A_I^*=H=\langle A_I\rangle$, so $A_I=A_{inv}$.

\printbibliography

@inproceedings{dong2024boundedclass,
  title={The Identity Problem in nilpotent groups of bounded class},
  author={Dong, Ruiwen},
  booktitle={Proceedings of the 2024 Annual ACM-SIAM Symposium on Discrete Algorithms (SODA)},
  pages={3919--3959},
  year={2024},
  organization={SIAM}
}

@InProceedings{dongUT4,
  author =	{Dong, Ruiwen},
  title =	{{On the Identity Problem for Unitriangular Matrices of Dimension Four}},
  booktitle =	{47th International Symposium on Mathematical Foundations of Computer Science (MFCS 2022)},
  pages =	{43:1--43:14},
  series =	{Leibniz International Proceedings in Informatics (LIPIcs)},
  ISBN =	{978-3-95977-256-3},
  ISSN =	{1868-8969},
  year =	{2022},
  volume =	{241},
  editor =	{Szeider, Stefan and Ganian, Robert and Silva, Alexandra},
  publisher =	{Schloss Dagstuhl -- Leibniz-Zentrum f{\"u}r Informatik},
  address =	{Dagstuhl, Germany},
  URL =		{https://drops.dagstuhl.de/entities/document/10.4230/LIPIcs.MFCS.2022.43},
  URN =		{urn:nbn:de:0030-drops-168415},
  doi =		{10.4230/LIPIcs.MFCS.2022.43}
}

@misc{romankov2022undecidability,
      title={Undecidability of the submonoid membership problem for a sufficiently large finite direct power of the Heisenberg group}, 
      author={Vitaly Roman'kov},
      year={2022},
      eprint={2209.14786},
      archivePrefix={arXiv},
      primaryClass={math.GR}
}

@InProceedings{ko_et_al,
  author =	{Ko, Sang-Ki and Niskanen, Reino and Potapov, Igor},
  title =	{{On the Identity Problem for the Special Linear Group and the Heisenberg Group}},
  booktitle =	{45th International Colloquium on Automata, Languages, and Programming (ICALP 2018)},
  pages =	{132:1--132:15},
  series =	{Leibniz International Proceedings in Informatics (LIPIcs)},
  ISBN =	{978-3-95977-076-7},
  ISSN =	{1868-8969},
  year =	{2018},
  volume =	{107},
  editor =	{Chatzigiannakis, Ioannis and Kaklamanis, Christos and Marx, D\'{a}niel and Sannella, Donald},
  publisher =	{Schloss Dagstuhl -- Leibniz-Zentrum f{\"u}r Informatik},
  address =	{Dagstuhl, Germany},
  URL =		{https://drops.dagstuhl.de/entities/document/10.4230/LIPIcs.ICALP.2018.132},
  URN =		{urn:nbn:de:0030-drops-91367},
  doi =		{10.4230/LIPIcs.ICALP.2018.132}
}

@book{clement2017theory,
  title={The theory of nilpotent groups},
  author={Clement, Anthony E and Majewicz, Stephen and Zyman, Marcos},
  volume={43},
  year={2017},
  publisher={Springer}
}

@article{baumslag1991algorithmic,
  title={The algorithmic theory of polycyclic-by-finite groups},
  author={Baumslag, Gilbert and Cannonito, Frank B and Robinson, Derek JS and Segal, Dan},
  journal={Journal of Algebra},
  volume={142},
  number={1},
  pages={118--149},
  year={1991},
  publisher={Academic Press}
}

@article{macdonald2015logspace,
  title={Logspace and compressed-word computations in nilpotent groups},
  author={Macdonald, Jeremy and Myasnikov, Alexei and Nikolaev, Andrey and Vassileva, Svetla},
  journal={arXiv preprint arXiv:1503.03888},
  year={2015}
}

@article{farkas_dax1997classroom,
  title={Classroom Note: An Elementary Proof of Farkas' Lemma},
  author={Dax, Achiya},
  journal={SIAM review},
  volume={39},
  number={3},
  pages={503--507},
  year={1997},
  publisher={SIAM}
}

@misc{bodart2024membership,
      title={Membership problems in nilpotent groups}, 
      author={Corentin Bodart},
      year={2024},
      eprint={2401.15504},
      archivePrefix={arXiv},
      primaryClass={math.GR}
}

@article{bell2010undecidability,
  title={On the undecidability of the identity correspondence problem and its applications for word and matrix semigroups},
  author={Bell, Paul C and Potapov, Igor},
  journal={International Journal of Foundations of Computer Science},
  volume={21},
  number={06},
  pages={963--978},
  year={2010},
  publisher={World Scientific}
}

\end{document}